\documentclass[12pt]{amsart}
\usepackage{verbatim}
\usepackage{amsmath,amssymb,epsfig,amscd,amsthm,xy}
\xyoption{all}

\newtheorem{theorem}{Theorem}[section]
\newtheorem{proposition}[theorem]{Proposition}

\newtheorem{lemma}[theorem]{Lemma}

\newtheorem{definition}[theorem]{Definition}

\theoremstyle{plain}
\numberwithin{equation}{theorem}

\theoremstyle{remark}
\newtheorem{remark}[theorem]{Remark}

\title[Singular Intersections]{Singular intersections of subgroups and character varieties}



\author[March\'e]{Julien March\'e}
\address{Institut de Math\'{e}matiques\\
Universit\'{e} Pierre et Marie Curie\\
75252 Paris c\'{e}dex 05\\
France}
\email{julien.marche@imj-prg.fr, }

\author[Maurin]{Guillaume Maurin}
\address{Institut de Math\'{e}matiques\\
Universit\'{e} Pierre et Marie Curie\\
75252 Paris c\'{e}dex 05\\
France}
\email{guillaume.maurin@imj-prg.fr, }
\begin{document}

\begin{abstract}
We prove a global local rigidity result for character varieties of 3-manifolds into $\rm{SL}_2$. Given a 3-manifold with toric boundary $M$ satisfying some technical hypotheses, we prove that all but a finite number of its Dehn fillings $M_{p/q}$ are globally locally rigid in the following sense: every irreducible representation $\rho:\pi_1(M_{p/q})\to\rm{SL}_2(\mathbb{C})$ is infinitesimally rigid, meaning that $H^1(M_{p/q},\textrm{Ad}_\rho)=0$. 

This question arose from the study of asymptotics problems in topological quantum field theory developed in \cite{cm2}. The proof relies heavily on recent progress in diophantine geometry and raises new questions of Zilber-Pink type. The main step is to show that a generic curve lying in a plane multiplicative torus intersects transversally almost all subtori of codimension 1. We prove an effective result of this form, based mainly on a height upper bound of Habegger.
\end{abstract}
\maketitle

\section{Introduction}\label{intro}
Let $M$ be a compact connected oriented 3-manifold without boundary. For any integer $k$ called level, the quantum Chern-Simons theory associated to the group SU$_2$ and the level $k$ gives an invariant $Z_k(M)\in \mathbb{C}$ called Witten-Reshetikhin-Turaev invariant. This invariant was introduced in \cite{witten} as a path integral, and constructed rigorously by Reshetikhin and Turaev using the representation theory of the quantum group $U_q sl_2$, see \cite{rt}. Formally, one can write

\[Z_k(M)=\int e^{ik\textrm{CS}(A)}\mathcal{D}A.\]
In this expression, $A$ is a 1-form on $M$ with values in the Lie algebra su$_2$ and 
\[\textrm{CS}(A)=-\frac{1}{4\pi}\int_M\rm{Tr}(A\wedge dA+\frac{2}{3}A\wedge [A\wedge A]).\]

The measure $\mathcal{D}A$ is of course ill-defined but Witten understood its cut-and-paste properties from which Reshitikhin and Turaev constructed the invariant rigorously. Applying formally the stationary phase expansion to this path integral, it localizes around the critical points of the Chern-Simons functional which correspond to the flat connections, that is 1-forms $A$ satisfying $dA+\frac{1}{2}[A\wedge A]=0$. The gauge equivalence classes of such connections correspond to conjugacy classes of representations $\rho:\pi_1(M)\to \mathrm{SU}_2$. Witten obtained formally the following asymptotic expansion:

\begin{equation}\label{conjwitten}
Z_k(M)=\sum_{\rho} e^{\frac{i\pi}{4}m(\rho)+ik \mathrm{CS}(\rho)}\sqrt{T(M,\rho)}+O(k^{-1/2}).
\end{equation}
In this formula, $\rho$ runs over the conjugacy classes of irreducible representations from $\pi_1(M)$ to SU$_2$, $m(\rho)$ is an element of $\mathbb{Z}/8\mathbb{Z}$ called spectral flow and $T(M,\rho)$ is the Reidemeister torsion of $M$ twisted by the representation Ad$_\rho$ of $\pi_1(M)$ on su$_2$.

This formula is proved in very few cases, one of the difficulties being that the Reidemeister torsion is defined only for those irreducible representations $\rho$ for which the space $H^1(M,\textrm{Ad}_\rho)$ vanishes. The space $H^1(M,\textrm{Ad}_\rho)$ can be identified to the Zariski tangent space of the character variety $X(M)$ at $\chi_\rho$ (see Section \ref{character}). Hence, a necessary condition for the Witten asymptotic formula to make sense is that the character variety is reduced of dimension 0. 

If $M$ is a compact, connected and oriented 3-manifold with toric boundary, we call Dehn surgery the result of gluing back to $M$ a solid torus. Let $\phi$ be a homeomorphism from $\partial M$ to $\partial D^2\times S^1$ reversing the orientation. It is well known that the homeomorphism type of the manifold $M\cup_\phi D^2\times S^1$ only depends on the homotopy class of the simple curve $\gamma=\phi^{-1}(S^1\times\{1\})\subset\partial M$. Hence we will denote by $M_\gamma$ this 3-manifold without boundary and call it the Dehn filling of $M$ with slope $\gamma$. 

In \cite{cm2}, Charles and the first author proved that the Witten conjecture holds for $M_\gamma$ if $M=S^3\setminus V(K)$ is the complement of a tubular neighborhood of the figure eight knot $K$, the linking number lk$(\gamma,K)$ is not divisible by 4 and the character variety $X(M_\gamma)$ is reduced of dimension 0. 
The strategy adopted in that article should generalize to any knot provided one has a strong version of the AJ-conjecture and some information on the Reidemeister torsion. 

The condition on the character variety to be reduced appeared as a technical point which was hard to check even in the case of the figure eight knot. However, we prove in this article that for a broad class of varieties $M$, this condition is satisfied for all but a finite number of slopes $\gamma$. More precisely, we show:

\begin{theorem}\label{theochar}
Suppose that $M$ is a compact connected oriented irreducible 3-manifold with toric boundary such that 
\begin{enumerate}
\item The map $r:X(M)\to X(\partial M)$ induced by the inclusion $\partial M\subset M$ is proper.
\item The character variety $X(M)$ is reduced.
\item The image by $r$ of the singular points of $X(M)$ are not torsion points of $X(\partial M)$ (see Section \ref{character}).
\end{enumerate}
Then for all but a finite number of slopes $\gamma$, the variety $X(M_\gamma)$ is reduced of dimension 0. Moreover, the number of exceptions can be effectively bounded. 
\end{theorem}

It is well-known that $X(\partial M)$ is the quotient of a 2-dimensional torus $\mathbb{G}_{\rm m}^2$ by the involution $\sigma(x,y)=(x^{-1},y^{-1})$. Denote by $\pi: \mathbb{G}_{\rm m}^2\to X(\partial M)$ the quotient. In the setting of the previous theorem, the variety $C=\pi^{-1}r(X(M))$ is a plane curve defined by the so-called $A$-polynomial, see \cite{ccgls}. The following notions will be central to the proof of Theorem \ref{theochar}.
\begin{definition}
Let $C$ and $C'$ be two curves in $\mathbb{G}_{\mathrm{m}}^2$. We say that $C$ intersects $C'$ transversally at $P\in C\cap C'$ if the two curves are smooth at $P$ with distinct tangent lines. We define the singular intersection $C\cap_{\mathrm{sing}} C'$ of $C$ and $C'$ as the set of all points $P\in C\cap C'$ where the two curves are smooth with equal tangent lines.
\end{definition}
For any couple of relatively prime integers $(p,q)$, let $H_{p,q}$ be the subtorus of $\mathbb{G}_{\mathrm{m}}^2$ defined by the equation $x^py^q=1$.  
Through some standard argument in character varieties, we reduce the proof of the previous theorem to show that $C$ intersects transversally $H_{p,q}$ for almost all $(p,q)$.\par  This fact is connected to recent questions in diophantine geometry surrounding the Zilber-Pink conjecture. Ineffectively, it follows from the 1999 bounded height property of Bombieri, Masser and Zannier (see Theorem 1 in \cite{Bombieri-Masser-Zannier-1}). Effective versions of the latter were worked out by Habegger over the years (see appendix B1 of \cite{Habegger-thesis}, Theorem 7 in \cite{Habegger-3} and \cite{Habegger-4}), allowing us to give an explicit upper bound on the maximal size of a couple $(p,q)$ such that $C$ has non-empty singular intersection with a translate of $H_{p,q}$.\par
 This upper bound might be of interest for the applications of Theorem \ref{theochar} in topology. It only depends on quantities that can be computed from an equation $f(x,y)=0$ defining $C$ in $\mathbb{G}_{\mathrm{m}}^2$. The polynomial $f$ is involved through its total and partial degrees and its logarithmic Weil height $h(f)$ (see section \ref{dio} for the definition).
 \begin{theorem}\label{plane_curves_intro}
 Let $C$ be a curve in $A$ with defining equation $f(x,y)=0$ for an irreducible polynomial $f\in\bar{\mathbb{Q}}[X,Y]$. Let $\delta=\mathrm{deg}(f)$, $\delta_x=\mathrm{deg}_x(f)$ and $\delta_y=\mathrm{deg}_y(f)$. Assume $C$ is not a translate of a subtorus. Then, for any translate $\gamma H_{p,q}$ with non-empty singular intersection with $C$, the quantity $\max(|p|,|q|)$ is at most $$\delta^3\,\mathrm{exp}\left((6.10^5+1)\delta^4\max\left(\delta_x\delta_y, h(f)\right)\right)\,.$$
 In particular, the union $C^{\{1\}}$ of all singular intersections of the form $C\cap_{\mathrm{sing}}H_{p,q}$ is a finite set.
 \end{theorem}
 We also prove a mild strengthening of the last sentence that looks like a perfect analogue of the Zilber-Pink conjecture in the context of plane singular intersections.\par
 In its multiplicative form --that is, when the ambient space is a multiplicative torus $T=\mathbb{G}_{\mathrm{m}}^n$-- the Zilber-Pink conjecture predicts what happens to a subvaritety $X$ when intersected to the union of all algebraic subgroups of fixed codimension $m$ (see \cite{Zilber} and \cite{Pink} for the original conjectures and \cite{Zannier} for a recent panorama of the subject).\par Under the assumption that $X$ is not contained in a proper algebraic subgroup of $T$, it is the statement that the subsets $X^{[m]}$ of $X$ defined by
 \begin{eqnarray}\label{zp}
 X^{[m]}=\bigcup_{\begin{array}{c} \mathrm{codim}\, H= m\\ \zeta\ \mathrm{torsion}\end{array}}X(\bar{\mathbb{Q}})\cap \zeta H(\bar{\mathbb{Q}})
 \end{eqnarray} are not Zariski-dense in $X$ for $m\ge \dim X+1$, where the union runs over all subtori $H$ of codimension $m$ and all torsion points $\zeta$ of $T$.\par Note that, in the particular case of a curve $C$ lying in $\mathbb{G}_{\mathrm{m}}^2$, the assumption on $C$ means precisely that $C$ is not a translate of a subtorus by a torsion point. In comparison with the hypotheses of Theorem \ref{plane_curves_intro}, this is weaker, yet it turns out to be sufficient for the finiteness of  $C^{\{1\}}$. Under this assumption, we can even prove the finiteness of a slightly larger subset of $C$. Its definition is derived from formula (\ref{zp}) for $C^{[1]}$ by changing all intersections for singular intersections.\par
 \begin{theorem}\label{tor_intro}
If $C$ is a curve in $\mathbb{G}_{\mathrm{m}}^2$ that is not a translate of a subtorus by a torsion point, then
$$C^{\{1,\mathrm{tor}\}}=\bigcup_{\begin{array}{c} p\wedge q  =  1\\ \zeta\ \mathrm{torsion}\end{array}}C(\bar{\mathbb{Q}})\cap_{\mathrm{sing}} \zeta H_{p,q}(\bar{\mathbb{Q}})$$ is a finite subset of $C$. 
 \end{theorem}

It is well known that, in the Zilber-Pink conjecture, the codimension value $m=\dim X+1$ is optimal for Zariski non-density. If $m$ is decreased further, then $X^{[m]}$ contains $X^{[\dim X]}$ that is dense in $X$ for all $X$. In this respect, the main feature of Theorem \ref{tor_intro} is to show that positive multiplicity of intersection can make up for a codimension drop among the $H$'s: going from $C^{[2]}$ to the larger subset $C^{[1]}$ generates infinitely many new points, but restricting to the case of positive multiplicity yields $C^{\{1,\mathrm{tor}\}}$ and finiteness is recovered.\par
This line of thought goes further than the case of plane curves and makes sense in a more general framework, leading to new conjectures of Zilber-Pink type. These generalizations will be addressed in a separate article.\par 
Finally, the last topic we study here is the relation between subsets of the form $C^{\{1,\mathrm{tor}\}}$ and $C^{[2]}$, showing that the first can be seen as a subset of the second type for a Zilber-Pink-like problem that takes place in a slightly different ambient space (see Theorem \ref{mixed I} and Remark \ref{mixed II}).\par
{\bf Acknowledgments:} We would like to thank L. Charles, D. Bertrand and P. Philippon for their kind interest and P. Habegger for guiding us through his work on effective bounded height properties. A different question related to character varieties of Dehn filling  was solved with the same kind of tools by B. Jeon in \cite{jeon}. We thank I. Agol for pointing it to us.

\section{Character variety and a reduction}\label{character}
Let $\Gamma$ be a finitely generated group. We denote by $R(\Gamma)$ the algebraic variety of all representations $\rho:\Gamma\to\rm{SL}_2(\bar{\mathbb{Q}})$. This variety is generally used over $\mathbb{C}$ by topologists whereas it is actually defined over $\mathbb{Z}$. We adopt here the field $\bar{\mathbb{Q}}$ which is more convenient for our purposes. The group SL$_2(\bar{\mathbb{Q}})$ acts on $R(\Gamma)$ by $g.\rho=g\rho g^{-1}$: we denote the algebraic quotient by $X(\Gamma)=R(\Gamma)//\rm{SL}_2(\bar{\mathbb{Q}})$. We refer to \cite{lm,cs} for the general theory and collect here some facts.

\begin{enumerate}
\item Given a representation $\rho\in R(\Gamma)$ we define its character $\chi_\rho:\Gamma\to\bar{\mathbb{Q}}$ by the formula $\chi_\rho(\gamma)=\rm{Tr}\,\rho(\gamma)$. As a set, $X(\Gamma)$ is the quotient of $R(\Gamma)$ by the relation $\rho\sim\rho'$ iff $\chi_\rho=\chi_{\rho'}$. This justifies the name {\it character variety}. 
\item  If $\rho,\rho'$ are two elements of $R(\Gamma)$ with $\chi_\rho=\chi_{\rho'}$ and $\rho$ irreducible, then $\rho$ and $\rho'$ are conjugated. 
\item The algebra of regular functions on $X(\Gamma)$ is generated by the so-called {\it trace functions} defined for any $\gamma\in \Gamma$ by $f_\gamma(\rho)=\rm{Tr}\,\rho(\gamma)$.
\item A representation $\rho\in R(\Gamma)$ is reducible if and only if for all $\alpha,\beta\in \Gamma$ one has $f_{[\alpha,\beta]}(\rho)=2$. In particular the set of reducible characters is Zariski-closed in $X(\Gamma)$ and is denoted by $X^{\rm red}(\Gamma)$ whereas its complement is denoted by $X^{\rm irr}(\Gamma)$. 
\item At an irreducible representation $\rho$, there is a natural isomorphism $T_{\chi_\rho}X(\Gamma)\simeq H^1(\Gamma,\textrm{Ad}_\rho)$. 
\item If $\Gamma=\mathbb{Z}^2$, we consider the morphism $\pi:\mathbb{G}_{\mathrm{m}}^2\to X(\Gamma)$ mapping $(x,y)$ to the character of the representation $\rho_{x,y}$ defined by 
\[\rho_{x,y}(a,b)=\begin{pmatrix}x^ay^b& 0 \\ 0 &x^{-a}y^{-b}\end{pmatrix}.\]
 It is well-known that $\pi$ induces an isomorphism between the quotient of $\mathbb{G}_{\mathrm{m}}^2$ by the involution $\sigma(x,y)=(x^{-1},y^{-1})$ and $X(\Gamma)$. In particular, we will denote by $X(\Gamma)_{\rm tor}$ the image by $\pi$ of the torsion points of $\mathbb{G}_{\mathrm{m}}^2$. 
\item If $\phi:\Gamma\to \Gamma'$ is a group homomorphism, it induces an algebraic morphism $\phi^*:X(\Gamma')\to X(\Gamma)$. 
\end{enumerate}

If $M$ is a connected compact oriented manifold, we set $X(M)=X(\pi_1(M))$. If $M$ is a surface or a 3-manifold as in Theorem \ref{theochar}, then it is an Eilenberg-Maclane space, which means that there is a natural isomorphism $H^1(\pi_1(M),\text{Ad}_\rho)\simeq H^1(M,\text{Ad}_\rho)$. Let $i:\partial M\to M$ be the inclusion morphism: it induces a map $i_*:\pi_1(\partial M)\to\pi_1(M)$. We denote by $r$ the map $(i_*)^*:X(M)\to X(\partial M)$ induced by the inclusion and call it the restriction map. 

If $M$ is a connected compact oriented 3-manifold with toric boundary, one can understand representations of $M_\gamma$ for a given slope $\gamma\subset \partial M$ in the following way: by Van-Kampen theorem, the fundamental group of $M$ is the amalgamated product $\pi_1(M)\underset{\pi_1(\partial M)}{*}\pi_1(D^2\times S^1)$. Moreover, the map $\pi_1(\partial M)\to\pi_1(D^2\times S^1)$ is surjective with kernel generated by $\gamma$ hence one has 
$\pi_1(M_{\gamma})=\pi_1(M)/\langle \gamma\rangle$ where $\langle \gamma\rangle$ is the normal closure of $\gamma$. 

In particular, a representation $\rho:\pi_1(M_{\gamma})\to \rm{SL}_2(\bar{\mathbb{Q}})$ is the same as a representation of $\rho:\pi_1(M)\to \rm{SL}_2(\bar{\mathbb{Q}})$ such that $\rho(\gamma)=1$. In terms of character varieties, $X(M_{\gamma})$ fits in the following diagram (which may not be cartesian):

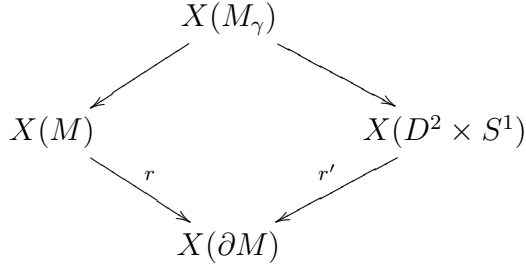
\[\xymatrix{& X(M_\gamma)\ar[dl]\ar[dr] & \\ X(M)\ar[dr]^r & & X(D^2\times S^1)\ar[dl]_{r'} \\ & X(\partial M) & }\]
The image of $r'$ is the projection of a subtorus of $\mathbb{G}_{\rm m}^2$ by the map $\pi$. We will reduce Theorem \ref{theochar} to considerations on the intersection of $\pi^{-1} r(X(M))$ with subtori of $\mathbb{G}_{\rm m}^2$. We start with a technical lemma.

\begin{lemma}
Let $M$ be a manifold satisfying the assumptions of Theorem \ref{theochar}, then every irreducible component of $X(M)$ has dimension 1. 
\end{lemma}
\begin{proof}
From now, we denote the local system Ad$_\rho$ with a subscript $\rho$. Let $Y$ be an irreducible (reduced) component of $X(M)$ and $\chi_\rho$ be a smooth point of it. A standard argument involving Poincar\'e duality (see \cite{HK} p. 42) shows that the rank of the map $i^*:H^1_\rho(M)\to H^1_\rho(\partial M)$ is half the dimension of $H^1_\rho(\partial M)$. By Poincar\'e duality again, $\textrm{rk }i^*=\dim H^0_\rho(\partial M)\in\{1,3\}$. As $r$ is proper, $r(Y)$ is a subvariety of the $2$-dimensional variety $X(\partial M)$. This shows that $r(Y)$ has dimension 1 and because $r$ is proper, $Y$ also has dimension 1. 
\end{proof}

\begin{proposition}
Let $M$ be a manifold satisfying the assumptions of Theorem \ref{theochar} and fix an homeomorphism between $\partial M$ and $S^1\times S^1$. A slope $\gamma$ corresponds to a pair $(p,q)$ of relatively prime integers.

Given $\chi_\rho$ a character of $X(M_{p/q})$, we denote by the same letter its restriction to $X(M)$. By the above remarks,  $r(\chi_\rho)=\pi(x,y)$ for some $(x,y)\in \bar{\mathbb{Q}}^2$ with $x^py^q=1$.

{\bf Case 1:} $x\ne \pm 1$ or $y\ne \pm 1$. 

 In that case, one can suppose that up to conjugation $\rho\circ i^*=\rho_{x,y}$. One has $H^1(M_{p/q},\text{Ad}_\rho)=(i^*)^{-1}T_{x,y}H_{p,q}$ where $i^*:H^1(M,\text{Ad}_\rho)\to H^1(\partial M,\text{Ad}_{\rho_{x,y}})$ is induced by the restriction map. 

In particular, if $\chi_\rho$ is a smooth point of $X(M)$ and $\pi^{-1}r(X(M))$ intersects $H_{p,q}$ transversally, then $H^1(M_{p/q},\text{Ad}_\rho)=0$. 

{\bf Case 2:} $x=\pm 1$ and $y= \pm 1$. 

If $\chi_\rho$ is a smooth point of $X^{\rm irr}(M)$, then $\rho$ factors through at most 1 Dehn filling $M_{p/q}$. 
\end{proposition}
\begin{proof}

The main point is a consequence of the Mayer-Vietoris sequence applied to the decomposition $M_{p/q}=M\cup D^2\times S^1$ given below. 
\begin{multline*}
H^0_\rho(M)\oplus H^0_\rho(D^2\times S^1)\to H^0_\rho(S^1\times S^1)\to \\
 H^1_\rho(M_{p/q})\to H^1_\rho(M)\oplus H^1_\rho(D^2\times S^1)\to H^1_\rho(\partial M)
 \end{multline*}
By the assumption $(x,y)\ne(\pm 1,\pm 1)$, the first map is onto and the result follows from the fact that the map $H^1_\rho(D^2\times S^1)\to H^1_\rho(S^1\times S^1)$ is the differential of the inclusion $H_{p,q}\subset A$.

In the second case, if $\chi_\rho$ is a smooth point of $X^{\rm irr}(M)$, then Poincar\'e duality implies that $H^0_\rho(\partial M)$ has dimension 1 and hence $\rho\circ i^*$ is a parabolic non-central representation. More explicitly, one has $\rho(a,b)=\pm \begin{pmatrix} 1 & au+bv \\ 0 & 1\end{pmatrix}$ for some $(u,v)\ne (0,0)$. Hence we can have $\rho(p,q)=1$ for at most one slope $[p:q]\in \mathbb{P}^1(\mathbb{Q})$ and the result follows.
\end{proof}

Set $C=\pi^{-1}(r(X(M))\subset \mathbb{G}_{\rm m}^2$. In the setting of Theorem \ref{theochar}, this is a curve defined by the so-called $A$-polynomial introduced in \cite{ccgls}.
Applying Theorem \ref{plane_curves} to $C$, we obtain that $C$ is transverse to $H_{p,q}$ at smooth points of $C$ for all but a finite number of slopes $(p,q)$. Moreover by assumption, singular points of $X(M)$ do not map to torsion points of $X(\partial M)$ and hence belong to at most one subtorus. The only remaining case is a singular point $(x,y)$ of $C$ which is not a singular value of $r$. In the neighborhood of $(x,y)$, $C$ is a union of branches with non-trivial tangents. Removing these tangents from the list of admissible $(p,q)$, we finally proved Theorem \ref{theochar}.

We would like to end this section with some comments on the topological meaning of the assumptions of Theorem \ref{theochar}. 
\begin{enumerate}
\item By Culler-Shalen theory (see \cite{shalen}), the properness assumption of $r:X(M)\to X(\partial M)$ is implied by the assumption that $M$ is small, meaning that it does not contain any closed incompressible surfaces (not boundary parallel). This assumption holds for a large family of knots such as 2-bridge knots. Such a hypothesis is necessary as the global local rigidity does not hold for instance for Whitehead doubles of knots. 
\item The reducibility of the character variety is a notoriously hard question. There is no reason to believe that the character variety of a knot complement in $S^3$ is reduced, however, we do not know any counter-example. 
\item The last assumption on singular points seems hard to check without knowing explicitly the character variety of $M$. However it is well-known that the singular points of $X(M)$ belonging to $X^{\rm red}(M)$ are encoded in the roots of the Alexander polynomial of $M$. Hence, the assumption contains in particular the fact that the Alexander polynomial does not vanish at roots of unity.
\end{enumerate}

\section{Plane curves}\label{dio}
This section is devoted to the proof of Theorem \ref{plane_curves_intro}. It relies heavily on the notions of {\em degree} and {\em height}.\par
Recall that the {\em degree} $\mathrm{deg}(P)$ of a point $P=(x_0:\ldots:x_n)$ of $\mathbb{P}_n(\bar{\mathbb{Q}})$ is defined as the minimal degree of a number field containing a system of homogeneous coordinates of $P$. Equivalently, assuming for simplicity that $x_0\neq 0$, $$\mathrm{deg}(P)=\left[\mathbb{Q}\left(\frac{x_1}{x_0},\ldots,\frac{x_n}{x_0}\right):\mathbb{Q}\right]\,.$$\par
The {\em logarithmic Weil height} $h(P)$ of $P$ is defined as follows. Let $K$ be a number field containing a system of homogeneous coordinates $(x_0,\ldots,x_n)$ of $P$. At any place $v$ of $K$, there is a unique absolute value $|\cdot|_v$ associated to $v$ such that $|p|_v\in\{1/p,1,p\}$, for any prime number $p$. Let $K_v$ ({\em respectively $\mathbb{Q}_v$}) be the completion of $K$ ({\em resp.} $\mathbb{Q}$) with respect to this absolute value ({\em resp.} the absolute value induced by $|\cdot|_v$). Then, $h(P)$ is given by the formula:
$$h(P)=\sum_v\frac{[K_v:\mathbb{Q}_v]}{[K:\mathbb{Q}]}\log\left(\max_{0\le i\le n} |x_i|_v\right)\,,$$
where the sum runs over all places of $K$. Because of the normalization factors $[K_v:\mathbb{Q}_v]/[K:\mathbb{Q}]$, the right-hand side neither depend on $K$, nore on the system of homogenous coordinates $(x_0,\ldots,x_n)$ that is chosen for $P$. Therefore, $h(P)$ is well-defined.\par Then, for any non-zero vector $v$ of $\bar{\mathbb{Q}}^{n+1}$, we define the {\em projective height of $v$} as the quantity $h(P_v)$, where $P_v$ is the point of $\mathbb{P}_n(\bar{\mathbb{Q}})$ with homogenous coordinates $v$. Finally, we define the height $h(f)$ of a polynomial $f$ with coefficients in $\bar{\mathbb{Q}}$ as the projective height of its vector of coefficients.

\par
 Roughly speaking, the height of a point measures its arithmetic complexity. For example if $P\in \mathbb{P}_1(\mathbb{Q})$ is a rational point written under the form $P=(p:q)$, where $p$ and $q$ are coprime integers, the definition above yields $$h(p:q)=\log(\max(|p|,|q|))\,.$$
In particular, for any positive real number $M$, there are only finitely many rational points $P$ of the projective line such that $h(P)\le M$.\par More generally, Northcott's theorem asserts that if the degree and height are bounded over a subset $S$ of $\mathbb{P}_n(\bar{\mathbb{Q}})$, then $S$ is a finite set. This crucial fact is one of the main feature of the height.\par As we will see, it can be used to obtain a second proof of Theorem \ref{plane_curves_intro}, except for the upper bound on $\max(|p|,|q|)$ which follows from a different route. This estimate will be derived from a height upper bound of Habegger through the following lemma.
\begin{lemma}\label{upper bound}
Let $f_1,f_2\in\bar{\mathbb{Q}}[X,Y]$ be polynomials of total degree $N_1$ and $N_2$ and let $N=\max(N_1,N_2)$. Let $\mathcal F$ be the family of coefficients appearing in either $f_1$ or $f_2$ and let $h(\mathcal{F})$ be the projective height of $\mathcal{F}$. Then, for any point $P=(x,y)$ of $\mathbb{A}^2(\bar{\mathbb{Q}})$,
$$h(f_1(P):f_2(P))\le Nh(x:y:1)+h(\mathcal{F})+\log \left(\begin{array}{c} N+2 \\2\end{array}\right)\,.$$ 
 \end{lemma}
 \begin{proof}
 Let $f_1(X,Y)=\sum_{k+\ell\le N_1} a_{k,\ell}X^kY^{\ell}$ and let $v$ be a place of a number field $K$ containing both coordinates of $P$ and the family $\mathcal{F}$. Then 
 \begin{eqnarray}\label{abs}
 \left|f_1(P)\right|_v\le\varepsilon(v)\max_{k,\ell}(|a_{k,\ell}|_v)\max(1,|x|_v,|y|_v)^{N_1}\,,
 \end{eqnarray}
 where
$$\varepsilon(v)= \left\{\begin{array}{cc}\left(\begin{array}{c} N_1+2 \\2\end{array}\right) &\ \mathrm{if}\ v\, |\, \infty\,, \\  \\ 1 &\ \mathrm{otherwise}\,. \end{array}\right.$$
 Assume for simplicity that $N_1\ge N_2$, so $f_2(X,Y)$ is again of the form $\sum_{k+\ell\le N_1} b_{k,\ell}X^kY^{\ell}$, for some $b_{k,l}\in\bar{\mathbb{Q}}$. Then an upper bound similar to (\ref{abs}) holds for $|f_2(P)|_v$. The only difference is the $a_{k,\ell}$'s are replaced by  the $b_{k,\ell}$'s. Therefore, $$\max(|f_1(P)|_v,|f_2(P)|_v)\le\varepsilon(v)\max\{|f|_v\mid f\in \mathcal{F}\}\max(1,|x|_v,|y|_v)^{N_1}$$
 and it follows that
 \begin{eqnarray*}
 h(f_1(P):f_2(P))&\le & N_1h(x:y:1)+h(\mathcal{F})+\sum_{v\, |\, \infty}\log(\varepsilon(v))\\
 &=& N_1h(x:y:1)+h(\mathcal{F}) +\log \left(\begin{array}{c} N_1+2 \\2\end{array}\right)\,,
 \end{eqnarray*}
 which concludes the proof as $N_1=\max(N_1,N_2)$.
 \end{proof}
 We are now ready to prove our upper bound. From now on, we let $A=\mathbb{G}_{\mathrm{m}}^2$ be the ambient multiplicative torus. Recall that, for any curve $C$ lying in $A$, we denote by $C^{\{1\}}$ the union of all singular intersections of the form $C\cap_{\mathrm{sing}} H_{p,q}$, where $(p,q)$ varies among all couples of relatively prime integers.
 \begin{theorem}\label{plane_curves}
 Let $C$ be a curve in $A$ with defining equation $f(x,y)=0$ for an irreducible polynomial $f\in\bar{\mathbb{Q}}[X,Y]$. Let $\delta=\mathrm{deg}(f)$, $\delta_x=\mathrm{deg}_x(f)$ and $\delta_y=\mathrm{deg}_y(f)$. Assume $C$ is not a translate of a subtorus. Then, for any translate $\gamma H_{p,q}$ with non-empty singular intersection with $C$, the quantity $\max(|p|,|q|)$ is at most $$\delta^3\,\mathrm{exp}\left((6.10^5+1)\delta^4\max\left(\delta_x\delta_y, h(f)\right)\right)\,.$$
 Moreover, $C^{\{1\}}$ is a finite set of effectively bounded degree and height.
 \end{theorem}
 \begin{proof}
 Let $P=(x,y)$ be a point belonging to the singular intersection of $C$ and a translate $\gamma H$, with $H=H_{p,q}$; then, $T_PC=T_P (\gamma H)$. The tangent space $T_P(\gamma H)$ is the subspace of $T_PA$ defined by the following equation $$\left(T_P(\gamma H)\right)\ :\ \frac{p}{x}dx+\frac{q}{y}dy=0\,.$$
Similarly,
$$(T_PC)\ :\  \,\frac{\partial f}{\partial x}(P)dx+\frac{\partial f}{\partial y}(P)dy=0\,.$$
Therefore, the equality $T_PC=T_P(\gamma H)$ means that the two vectors of partial derivatives $$\left(\frac{p}{x},\frac{q}{y}\right)\quad \mathrm{and}\quad \left(\frac{\partial f}{\partial x}(P),\frac{\partial f}{\partial y}(P)\right)$$ are colinear.\par Using height theory, we now derive that $\max (|p|,|q|)$ is bounded from above, thereby showing that the possible $H$'s are finitely many. Indeed, $p$ and $q$ are coprime, so $\log( \max (|p|,|q|))=h(p:q)$ is the logarithmic Weil height of the point $(p:q)\in\mathbb{P}^1(\bar{\mathbb{Q}})$. Then, colinearity of the vectors of partial derivatives means that the corresponding points of $\mathbb{P}^1(\bar{\mathbb{Q}})$ are equal or, equivalently, $$\left(p:q\right)=\left(x\frac{\partial f}{\partial x}(P):y\frac{\partial f}{\partial y}(P)\right)\,.$$ 

Hence, using lemma \ref{upper bound} with $f_1=x\,\partial f/\partial x$ and $f_2=y\,\partial f/\partial y$, we obtain the following upper bound:
\begin{eqnarray*}
h(p:q)\le \delta h(x:y:1)+h(\mathcal{F}) +\log\left(\frac{(\delta+1)(\delta+2)}{2}\right)\,,
\end{eqnarray*}
where $\mathcal{F}$ is the family of coefficients of the partial derivatives polynomials $\partial f/\partial x$ and $\partial f/\partial y$.\par Here, we can obtain a slightly better $\mathrm{log}$-term, with numerator $\delta(\delta+1)$ instead of $(\delta+1)(\delta+2)$, by using the relation $h(p:q)=h_{\mathrm{m}}(p/q)$ where the right-hand side is the height of the point $p/q$ of $\mathbb{G}_{\mathrm{m}}(\bar{\mathbb{Q}})$. This height satisfies a triangle inequality relative to multiplication, so
\begin{eqnarray*}\label{tor}
h_{\mathrm{m}}\left(\frac{p}{q}\right)&\le& h_{\mathrm{m}}\left(\frac{p}{q}\,\frac{y}{x}\right) + h_{\mathrm{m}}\left(\frac{x}{y}\right)\\
&= & h\left(\frac{p}{x}:\frac{q}{y}\right)+ h(x:y)\\
&\le &h\left(\frac{\partial f }{\partial x}(P):\frac{\partial f }{\partial y}(P)\right) + h(P)
\end{eqnarray*}
and then we apply our lemma to the partial derivatives of $f$ instead of $x\,\partial f/\partial x$ and $y\,\partial f/\partial y$. We will use this sharper bound in the sequel.\par Finally, any element of $\mathcal{F}$ is a product $k a_{\alpha}$ of a coefficient of $f(X)=\sum_{\alpha}a_{\alpha}x^{\alpha_1}y^{\alpha_2}$ and a positive integer $k$ equal to $\alpha_1$ or $\alpha_2$ pending on which partial derivative of $f$ is considered. In any case, we have $k \le \max(\delta_x,\delta_y)\le \delta$, so $$h(\mathcal{F})\le h(f)+\log\delta$$ and therefore
\begin{eqnarray}\label{upI}
\qquad h(p:q)\le \delta h(x:y:1)+h(f)+\log\left(\frac{\delta^2(\delta+1)}{2}\right)\,.
\end{eqnarray}
\par To conclude the proof, we use an explicit height upper bound for the points of $C^{[1]}$ due to Habegger (see Theorem B.1 in \cite{Habegger-thesis}). It reads as follows: for any point $P=(x,y)$ in $C^{[1]}$, 
\begin{eqnarray}\label{eff_bhc}
\max(h_{\mathrm{m}}(x),h_{\mathrm{m}}(y)) \le 3.10^5\,\delta^3\, \max\left(\delta_x\delta_y, h(f)\right)\,.
\end{eqnarray}
Combining this result with inequality (\ref{upI}) and using the elementary fact that $h(x:y:1)\le h_{\mathrm{m}}(x)+h_{\mathrm{m}}(y)$, we obtain
\begin{eqnarray*}
h(p:q)&\le & 6.10^5\delta^4\max\left(\delta_x\delta_y, h(f)\right)+h(f)+\log\left(\frac{\delta^2(\delta+1)}{2}\right)\\
 & \le&(6.10^5+1)\delta^4\max\left(\delta_x\delta_y, h(f)\right) +\log(\delta^3)
\end{eqnarray*}
and taking exponentials of both sides gives the upper bound of the theorem.\par
Finally, boundedness of $h(P)$ implies boundedness of $\max(|p|,|q|)$, which proves finiteness of the $H$'s. Restricting to the case $\gamma=1$, {\em i.e.} $P\in C^{\{1\}}$, it implies finiteness of the $P$'s, because the assumption on $C$ guarantees that its intersection with any proper subtorus is finite.\par Of course, the effective upper bound on $h(P)$ mentioned in our theorem is Habegger's bound (\ref{eff_bhc}). Whereas the effective upper bound on $\mathrm{deg}(P)$ follows from our bound on $\max(|p|,|q|)$ by applying B\'ezout's theorem to the intersection of $C$ and $H_{p,q}$, which gives $$\mathrm{deg}(P)\ll \max(|p|,|q|)\mathrm{deg}(C)[K_0:\mathbb{Q}]\,,$$ where $K_0$ is any number field such that $C$ is defined over $K_0$ (see the proof of lemma \ref{NoP} below for a detailed exposition of a similar argument).
\par
 \end{proof}
Theorem \ref{plane_curves} is non trivial in the sense that both $C^{[1]}$ and $$C^{\{1, A(\bar{\mathbb{Q}})\}}=\bigcup_{p\wedge q=1} \{P\in C_{\mathrm{reg}}(\bar{\mathbb{Q}})\mid T_PC=T_P(PH_{p,q})\}$$ are infinite for all $C$. The first is obtained from $C^{\{1\}}$ by removing the tangency condition, while the second is the set of smooth points $P\in C(\bar{\mathbb{Q}})$ at which $T_PC$ coincides with the tangent line of a translated $H_{p,q}$ -- without any multiplicative dependance assumption on $P$.\par The case of $C^{[1]}$ is well known. One of the two projections $\pi$ induces a dominant morphism over $C$, so $\pi(C)$ contains a dense open subset of $\mathbb{G}_{\mathrm{m}}$. In particular, it contains infinitely many torsion points and their inverse images form an infinite subset of $C^{[1]}$. \par For $C^{\{1, A(\bar{\mathbb{Q}})\}}$, the argument is completely similar, up to a different choice of the dominant morphism $\pi$. Assuming $f(x,y)=0$ is a defining equation for $C$, with $f$ irreducible, let $\pi=\sigma_C$ be the rational map from $C$ to the dual projective line which sends any smooth point $P\in C(\bar{\mathbb{Q}})$ to the point 
\begin{eqnarray}\label{tan}
\sigma_C(P)=\left(x \frac{\partial f}{\partial x}(P): y \frac{\partial f}{\partial y}(P)\right)
\end{eqnarray}
 of $\mathbb{P}_1^*$. Note that it doesn't depend on the irreducible polynomial $f$ chosen initially. Moreover, $\sigma_C$ has the following property: for any smooth point $P=(x,y)$ of $C(\bar{\mathbb{Q}})$, $$\sigma_C(P)=(\alpha:\beta)\Longleftrightarrow(T_PC)\,:\ \alpha\frac{dx}{x}+\beta\frac{dy}{y}=0\,.$$  It follows that $\sigma_C$ is constant if and only if $C$ is a translated subtorus, in which case $\sigma_C(C)$ is a rational point and $C^{\{1, A(\bar{\mathbb{Q}})\}}=C$. Otherwise, $\sigma_C$ is dominant, so its image contains a dense open subset of the line, hence infinitely many rational points. Finally, the inverse images of these points in $C$ form an infinite subset of $C^{\{1, A(\bar{\mathbb{Q}})\}}$.\par
Hence,  $C^{\{1,A(\bar{\mathbb{Q}})\}}$ is always Zariski-dense in $C$. However, under the assumption of Theorem \ref{plane_curves}, it is a sparse subset in the following sense. 
\begin{lemma}\label{NoP}
Let $C$ be a curve in $A$ defined over a number field $K_0$ and assume $C$ is not a translate of a subtorus. Then, for all $P\in C^{\{1, A(\bar{\mathbb{Q}})\}}$, $$\mathrm{deg}(P)\le \mathrm{deg}(C)^2[K_0:\mathbb{Q}]\,.$$
\end{lemma}
\begin{proof}
Assume first that $C$ is defined over $\mathbb{Q}$, {\em i.e.} $C$ can be defined by an equation $f(x,y)=0$ with coefficients in $\mathbb{Q}$. Moreover, assume $f $ is irreducible over $\bar{\mathbb{Q}}$, so $\mathrm{deg}(f)=\mathrm{deg}(C)$.\par From formula (\ref{tan}), it follows that the rational map $\sigma_C$ is also defined over $\mathbb{Q}$. Hence, so is its fiber over any rational point $(p:q)$ of $\mathbb{P}^*_1$. Therefore, if $P=(x,y)$ belongs to $\sigma_C^{-1}(p:q)$, then so does any Galois conjugate of $P$ over $\mathbb{Q}$. Moreover, the number of distinct conjugates of $P$ is precisely the number of embeddings of $\mathbb{Q}(x,y)$ in $\bar{\mathbb{Q}}$ over $\mathbb{Q}$,  which equals the degree of $P$. Hence, $$\mathrm{deg}(P)\le |\sigma_C^{-1}(p:q)|$$ and the lemma thus follows from the estimate
\begin{eqnarray}\label{degP}
|\sigma_C^{-1}(p:q)|\le \mathrm{deg}(C)^2\,.
\end{eqnarray}
To prove (\ref{degP}), consider the irreducible components $C_1,\ldots, C_r$ of the algebraic subset of $A$ defined by the vanishing of
$$g_{p,q}(x,y)=qx \frac{\partial f}{\partial x}(x,y)-py \frac{\partial f}{\partial y}(x,y)\,.$$
Then, $\sigma_C^{-1}(p:q)$ is the union of the intersections $C_{\mathrm{reg}}\cap C_i$ and $$\sum_{i=1}^r\mathrm{deg}(C_i)\le\mathrm{deg}(g)\le \mathrm{deg}(f) =\mathrm{deg}(C)\,.$$ Using B\'ezout's theorem, it follows that $$|\sigma_C^{-1}(a:b)|\le \sum_{i=1}^r\mathrm{deg}(C)\mathrm{deg}(C_i)\le \mathrm{deg}(C)^2\,,$$
which proves the claim if $C$ is defined over $\mathbb{Q}$.\par In general, $C$ is defined over a number field $K_0$. Considering conjugates over $K_0$ instead of conjugates over $\mathbb{Q}$, the argument above yields $$ \mathrm{deg}(P)\le [K_0(x,y):K_0][K_0:\mathbb{Q}]\le |\sigma_C^{-1}(p:q)|[K_0:\mathbb{Q}]$$
and the lemma follows from (\ref{degP}).
 \end{proof}
\begin{remark}
In comparison with Theorem \ref{plane_curves}, this lemma provides a much sharper bound for the degrees of the points of $C^{\{1\}}$. Recall that the one from the theorem was derived from the fact that $$\mathrm{deg}(P)\ll \max(|p|,|q|)\mathrm{deg}(C)[K_0:\mathbb{Q}]$$ combined with our estimate on $\max(|p|,|q|)$, thereby leading to a large upper bound depending on the height of $C$.
\end{remark}
Applying Northcott's theorem, it follows from boundedness of the degree that any subset of $C^{\{1, A(\bar{\mathbb{Q}})\}}$ of bounded height is finite. We will refer to this as a {\em Northcott property}. In particular, it implies that $$C^{\{1\}}\subset C^{[1]}\cap C^{\{1,A(\bar{\mathbb{Q}})\}}$$ is a finite set, which gives a second proof of Theorem \ref{plane_curves}, except for the upper bound on $\max(|p|, |q|)$.\par
This second approach leads to the following result. Recall that a torsion variety is a translate of a subtorus by a torsion point.
 \begin{theorem}\label{tor}
Let $C$ be a curve in $A$ that is not a torsion variety. Then, the union $C^{\{1,\mathrm{tor}\}}$ of all singular intersections of $C$ with 1-dimensional torsion subvarieties of $A$ is a finite set of effectively bounded height. 
 \end{theorem}
 \begin{proof}
Assume first that $C$ is not a translate of a subtorus. Then, the height is bounded over $C^{[1]}$ (see Theorem 1 in \cite{Bombieri-Masser-Zannier-1}), hence also on its subset $C^{\{1,\mathrm{tor}\}}$. As the latter is also a subset of $C^{\{1,A(\bar{\mathbb{Q}})\}}$, the Northcott property derived from lemma \ref{NoP} proves the claim.\par Now, assume $C$ is a translate $\gamma H$ of a subtorus $H$ of $A$. Then, because of the assumption on $C$, $\gamma$ is non-torsion. Hence $C\neq \zeta H$ for all $\zeta\in A_{\mathrm{tor}}$, so $C\cap \zeta H=\emptyset$ because the two are distinct translates of the same subtorus. Finally, if $H'\neq H$ is a second subtorus, then $T_P(PH')\cap T_P(PH)=0$ at any point $P\in A(\bar{\mathbb{Q}})$. Therefore, $C\cap_{\mathrm{sing}}\zeta H'$ is again empty, which gives $C^{\{1,\mathrm{tor}\}}=\emptyset$ and completes the proof of the theorem.
 \end{proof}
\par 
We conclude this section with an alternative formulation of Theorem \ref{tor} showing the equivalence of this statement to a tangential Zilber-Pink problem. It also explains the analogy between subsets of the form $C^{\{1,\mathrm{tor}\}}$ and $C^{[2]}$: the first correspond to a subset of the second form by considering a section over $C$ of the dual projectivized tangent bundle $\mathbb{P}^*(TA)$. Recall that the fiber of $\mathbb{P}^*(TA)$ over a point $P\in A(\bar{\mathbb{Q}})$ parametrizes lines in $T_PA$.\par For $C$ smooth, the section of interest $\mathcal{C}$ is simply $$\mathbb{P}^*(TC)=\{(P,[T_PC])\mid P\in C(\bar{\mathbb{Q}})\}\,.$$
When $C$ is singular, $\tilde{\sigma}_C(P)=(P,[T_PC])$ only defines a rational section of $\mathbb{P}^*(TA)$ over $C$. We thus define $\mathcal{C}$ as the Zariski closure of the image of $\tilde{\sigma}_C$ in $\mathbb{P}^*(TA)$. We call $\mathcal C$ the {\em tangent section} of $\mathbb{P}^*(TA)$ over $C$.\par
To make the connection between $\tilde{\sigma}_C$ and the $\sigma_C$ defined previously, consider first the trivialization of the cotangent bundle $T^*A$ ({\em resp}. the tangent bundle $TA$) associated to the global 1-forms $dx/x$ and $dy/y$ ({\em resp.} the global vector fields $x\partial/\partial x$ and $y\partial/\partial y$). Through this trivialization, the map $\pi:\mathbb{P}^*(TA)\to A$ can be identified to the first projection $A\times \mathbb{P}_1^*\to A$ and our rational section $\tilde{\sigma}_C$ is then given by $\tilde{\sigma}_C (P)=(P,\sigma_C(P))$.\par
Finally, for any couple of relatively prime integers $(p,q)$ and any torsion point $\zeta$ of $A$, let $\mathcal{H}_{p,q}^{\,\zeta}$ be the tangent section of $\mathbb{P}^*(TA)$ over the torsion variety $\zeta H_{p,q}$.
\begin{theorem}\label{mixed I}
Let $C$ be a curve in $A$ and let $\mathcal C$ be the tangent section of $\mathcal{A}=\mathbb{P}^*(TA)$ over $C$. If $\mathcal{C}$ is not of the form $\mathcal{H}_{p,q}^{\,\zeta}$, then$$\mathcal{C}^{[2]}= \bigcup_{\begin{array}{c} p\wedge q  =  1\\ \zeta\ \mathrm{torsion}\end{array}} \mathcal{C}(\bar{\mathbb{Q}})\cap \mathcal{H}_{p,q}^{\,\zeta}$$ is a finite set.
\end{theorem}
\begin{remark}\label{mixed II}
In the formula above, intersections are usual intersections, not singular ones. Moreover $\mathrm{codim}(\mathcal{H}_{p,q}^{\,\zeta},\mathcal{A})=2$, so the union of the theorem is indeed a subset of the type $\mathcal{C}^{[2]}$ for a Zilber-Pink-like problem in $\mathcal A$.
\end{remark}
\begin{proof}
It suffices to show that there are finitely many smooth points $P$ of $C$ such that $\tilde{\sigma}_C(P)=(P,[T_PC])$ lands in the union of the $\mathcal{H}_{p,q}^{\,\zeta}$. But assuming $\tilde{\sigma}_C(P)\in\mathcal{H}_{p,q}^{\,\zeta}$ means precisely that there is a point $Q$ in $C'=\zeta H_{p,q}$ such that $$\tilde{\sigma}_C(P)=\tilde{\sigma}_{C'}(Q)\,.$$
In other words, $P=Q$ and $C$ and $C'$ have equal tangent lines at $P$, so $P\in C\cap_{\mathrm{sing}}C'$; finiteness thus follows from Theorem \ref{tor}.
\end{proof}

\end{document}